\documentclass[11pt]{amsart}

\usepackage{amssymb,amsmath,enumerate,epsfig}
\usepackage{amsfonts,color}
\usepackage{a4,latexsym}
\usepackage{parskip}

\usepackage{hyperref}
\hypersetup{pdftitle=DK11}

\newcommand{\C} {\mathbb{C}}
\newcommand{\Q} {\mathbb{Q}}

\newcommand{\F}{\mathbb{F}}
\newcommand{\Z}{\mathbb{Z}}

\newcommand{\eps}{\varepsilon}

\newcommand{\PP}{\mathbb{P}}
\newcommand{\NS}{\mathop{\rm NS}}

\newcommand{\Br}{\mathop{\rm Br}\nolimits}

\newcommand{\het}[1]{H_\text{\'et}^2(#1,\Q_l)}
\newcommand{\tr}{\mathop {\rm tr}}
\newcommand{\Fr}{\mathop {\rm Frob}}

\newtheorem{Theorem}{Theorem}[section]
\newtheorem{Proposition}[Theorem]{Proposition}

\newtheorem{Corollary}[Theorem]{Corollary}

\theoremstyle{remark}
\newtheorem{Remark}[Theorem]{Remark}

\theoremstyle{definition}

\begin{document}

\title{K3 surfaces with an automorphism of order 11}


\author{Matthias Sch\"utt}
\address{Institut f\"ur Algebraische Geometrie, Leibniz Universit\"at
  Hannover, Welfengarten 1, 30167 Hannover, Germany}
\email{schuett@math.uni-hannover.de}
\urladdr{http://www.iag.uni-hannover.de/schuett/}

\thanks{Funding   by ERC StG~279723 (SURFARI)
 is gratefully acknowledged}

\date{September 30, 2013}

 \begin{abstract}
This paper concerns K3 surfaces with automorphisms of order 11 in arbitrary characteristic.
Specifically we study the wild case and prove that a generic such surface in characteristic 11
has Picard number 2.
We also construct K3 surfaces with an automorphism of order 11 in every characteristic,
and supersingular K3 surfaces whenever possible.
 \end{abstract}
 \maketitle


%

\section{Introduction}

Automorphisms of K3 surfaces have gained considerable attention in recent years.
Over $\C$ they usually lend themselves to lattice theoretic investigations
while positive characteristic often requires more delicate considerations.
Order 11 is special in the sense that 
by \cite{N} it is the first prime order which over $\C$  does not allow symplectic automorphisms 
(leaving the regular 2-form invariant).
This result carries over to any positive characteristic
except for $p=11$.
In fact, it is this wild case which our primary interest lies in.
The wild case was studied in great detail in \cite{DK},
and our main aim is to solve the problem of the Picard number which was left open there (see \cite[Rem.~3.6]{DK}):

\begin{Theorem}
\label{thm}
Let $k$ denote an algebraically closed field of characteristic $11$.
Then a generic K3 surface with an automorphism of order $11$ over $k$
has Picard number $\rho=2$.
\end{Theorem}

K3 surfaces with a wild automorphism of order $11$
come in 1-dimensional families
where each surface is a torsor of a jacobian elliptic K3 surface, see Section \ref{s:background}.
%
This allows us to reduce
the proof of Theorem \ref{thm} to the  family of jacobian elliptic K3 surfaces 
with a wild automorphism 
of order $11$ as studied in \cite{DK}.
Then it suffices to exhibit a member of this family with Picard number $\rho=2$.
Our method to achieve this is arguably elementary 
as we simply compute the zeta functions of the elliptic  K3 surfaces over $\F_{11}$ in question.
In contrast, the way to derive the zeta function in an effective fashion is quite non-elementary, 
as it makes essential use of the automorphisms.
As a by-product we also compute the heights of the K3 surfaces over $\F_{11}$
which attain the finite maximum $h=10$ 
(Corollary \ref{Cor:h}).

The paper is organised as follows.
After reviewing the background, particularly from \cite{DK}, in Section \ref{s:background},
we give the proof of Theorem \ref{thm}
in Section \ref{s:proof}.
The paper concludes with considerations for other characteristics in Section \ref{s:p}:
Proposition \ref{prop} states that any characteristic $p$ admits K3 surfaces with automorphisms of order $11$,
and whenever possible (i.e.~$p=11$ or there is some $\nu$ such that $p^\nu\equiv -1\mod 11$), 
there are supersingular examples. 
This answers two questions from \cite[Rem.~6.6]{DK}
in a surprisingly uniform manner.

\section{Background}
\label{s:background}

Let $k$ denote an algebraically closed field of characteristic $p=11$.
Let $X$ be a K3 surface over $k$ admitting an automorphism $\varphi$ of order $11$.
Then by \cite[Thm.~1.1]{DK}
$X$ admits an elliptic fibration that is compatible with $\varphi$ 
and its induced action on $\varphi'\PP^1$.
That is, there is a commutative diagram
$$
\begin{array}{ccc}
X & \stackrel{\varphi}\longrightarrow & X\\
\downarrow && \downarrow\\
\PP^1 &  \stackrel{\varphi'}\longrightarrow & \PP^1
\end{array}
$$
In fact, $X$ is $\langle\varphi\rangle$-equivariantly isomorphic
to a  torsor over a specific jacobian elliptic K3 surface $X_\eps$
which can be given as
\begin{eqnarray}
\label{eq:Xe}
X_\varepsilon:\;\;
y^2 = x^3 + \varepsilon x^2 + t^{11} - t \;\;\; (\varepsilon\in k).
\end{eqnarray}
The K3 surface $X_\varepsilon$ admits an automorphism of order $11$ given by
\begin{eqnarray}
\label{eq:phi}
\varphi: (x,y,t) \mapsto (x,y,t+1).
\end{eqnarray}
Generically the elliptic fibration over $\PP^1_t$
has the following singular fibres:
a cuspidal cubic at $t=\infty$
and 22 nodal cubics at the roots of the affine discriminant
\[
\Delta=-(t^{11}-t)(5t^{11}-5t+4\varepsilon^3).
\]
Visibly, for $\varepsilon=0$, the discriminant degenerates to a square,
and the nodal cubics degenerate to cuspidal ones.
In fact, the special member $X_0$ is isomorphic to the Fermat quartic over $k$.
Here's a conceptual line of argument to see this
which we will refer to again in the proof of Proposition \ref{prop}.
To start with, $X_0$ arises from the singular K3 surface 
\[
X: \;\; y^2 = x^3 + t^{11} - 11 t^6 -t
\]
over $\Q$ by way of reduction.
Here the term {\it singular} refers to the Picard number attaining the maximum of $20$ over $\C$,
and it was proved in \cite{shioda} that $X$ is in fact singular of discriminant $-300$.
In particular, this implies that $X$ is modular,
i.e.~associated to $X$ there is a modular form $f$ with complex multiplication by $\Q(\sqrt{-3})$.
By virtue of the vanishing of the Fourier coefficients of $f$,
$X$ has supersingular reduction at any prime $p_0\equiv -1\mod 3$:
\[
\rho(X\otimes\bar{\F}_{p_0})=22, \;\;\; \text{ and in particular, } \;\; \rho(X_0)=22.
\]
Then by \cite[Proposition 1.0.1]{Shimada} any of the supersingular reductions has Artin invariant $1$.
Another such example is the Fermat quartic in any characteristic $p_1\equiv -1\mod 4$.
But then supersingular K3 surfaces with Artin invariant 1 are unique up to isomorphism over an algebraically closed field by \cite{Ogus}.
Hence $X_0$ is isomorphic to the Fermat quartic over $k$.

Note how we derived in a rather indirect way the Picard number of $X_0$.
In comparison, the Picard number of $X_\eps$ does not seem to be known for $\eps\neq 0$.
In this paper, we will prove that generically 
\[
\rho(X_\eps)=2
\]
by verifying the claim for all $\eps\in\F_p^\times$.
Our proof will make use of Lefschetz' fixed point formula for $l$-adic \'etale cohomology ($l\neq p$).
In the remainder of this section,
we review some basic results from \cite{DK}
and set up the notation.

A key step in \cite{DK} is to show that
\[
\het{X_\eps}^{\varphi^*=1} = U\otimes\Q_l,
\]
where $U$ is the hyperbolic plane generated by general fibre $F$ and zero section $O$.
In particular, the orthogonal complement with respect to cup-product,
\[
V=U^\bot\subset\het{X_\eps},
\]
decomposes into 10 equidimensional eigenspaces for $\varphi^*$.
With $\zeta$ a primitive eleventh root of unity in $\bar\Q_l$,
we obtain the 2-dimensional $\Q_l(\zeta)$-vector spaces
\[
V_i = \het{X_\eps}^{\varphi^*=\zeta^i} \;\;\; (i=0,\hdots,10)
\]
such that $V=\oplus_{i=1}^{10} V_i$ and $V_0 = U\otimes\Q_l$. 
This eigenspace decomposition can be used to show
that $\rho(X_\eps)=2,12$ or $22$.
In fact, there is a simple argument to prove that $\rho\neq 12$
for infinitely many $\varepsilon\in k$.
Namely $X_\eps$ admits a $k$-isomorphic model
\begin{eqnarray}
\label{eq:gamma}
X_\gamma:\;\;\; 
y^2 = x^3 + \gamma x+ t^{11}-t
\end{eqnarray}
which comes equipped with the automorphism $\varphi$ of order $11$ as in \eqref{eq:phi}
and with an automorphism $\imath$ of order $4$:
\begin{eqnarray}
\label{eq:imath}
\imath: \;\; (x,y,t) \mapsto (-x,\sqrt{-1} y,-t).
\end{eqnarray}
Since the automorphism $\imath$ does not commute with $\varphi$,
its impact does not translate directly onto the cohomology eigenspaces $V_i$.
There is, however, an indirect argument for any $X_\gamma$ 
defined over some 
field $\F_q$ with $q=p^r, r$ odd.
Namely, because of the symmetry encoded in \eqref{eq:imath},
the pair $(x,t)\in\F_q^2$ gives a non-zero square in $\F_q$ upon substituting in the right-hand side of \eqref{eq:gamma} if and only if $(-x,-t)$ gives a non-square.
That is to say, the numbers of $\F_q$-rational points 
on the fibres $F_t$ and $F_{-t}$ of $X_\gamma$ add up to a constant number:
\[
\# F_t(\F_q) + \# F_{-t}(\F_q) = 2q+2 \;\;\; (t\in\PP^1(\F_q)).
\]
As a consequence, we deduce that
\[
\# X_\gamma(\F_q) = (q+1)^2.
\]
In comparison, Lefschetz' fixed point formula gives, with the decomposition $\het{\bar X_\eps} = U\otimes\Q_l \oplus V$, 
\begin{eqnarray}
\label{eq:lef}
\;\;\;\;
\# X(\F_q) = 1 + \tr {\Fr}^*_q(\het{\bar X_\eps}) + q^2 = 1 + 2q + \tr {\Fr}^*_q(V) + q^2.
\end{eqnarray}
Thus we deduce that 
\begin{eqnarray}
\label{eq:0}
\tr {\Fr}^*_q(V) = 0 \;\;\; \text{for any }\; q=p^r, r \text{ odd}.
\end{eqnarray}
Denote by $\mu_q(T)$  the characteristic polynomial  of $\Fr^*_q$ on $V$.
Then \eqref{eq:0} implies
\[
\mu_q(T) = \nu_q(T^2),\;\;\; \nu_q(T)\in\Z[T],\; \deg(\nu_q)=10.
\]
In fact we deduce right away that
\[
\mu_{q^2}(T) = \nu_q(T)^2.
\]
Now one can argue with the roots of this characteristic polynomial
as in the following section
and appeal to the Tate conjecture to rule out the alternative $\rho(X_\gamma)=12$.

\section{Proof of Theorem \ref{thm}}
\label{s:proof}

Essentially our proof of Theorem \ref{thm}
amounts to computing the characteristic polynomials $\mu_p$ of $\Fr^*_p$ on $V$
for (any of) the surfaces $X=X_\gamma$ or $X_\eps$ for $\gamma, \eps\in\F_p^\times$.
Then we use that the Galois group always acts through a finite group on $\NS(X)$,
so that
via the cycle class map $\NS(X)\otimes\Q_l(-1)\hookrightarrow\het{\bar X}$
\begin{eqnarray}
\label{eq:Tate}
\;\;\;\;
\rho(X) \leq \#\{\text{zeroes of $\mu_p(T)$ of the shape $\xi p$ for some root of unity $\xi$}\}.
\end{eqnarray}
Here equality is subject to the Tate conjecture \cite{Tate}
which holds for elliptic K3 surfaces by \cite{ASD}.
Without having to appeal to the Tate conjecture,
we will find that the right-hand side of \eqref{eq:Tate} gives an upper bound of $2$ for $\rho(X)$,
thus infering $\rho(X)=2$ by virtue of the classes of fibre and zero section generating $U$.
This will enable us to prove Theorem \ref{thm} for the family of jacobian elliptic K3 surfaces
with an automorphism of order $11$.
From this we will infer the full statement of Theorem \ref{thm} in \ref{ss:full}.

\subsection{Strategy}
We shall now outline the strategy how to compute the characteristic polynomials of $\Fr^*_p$
quite easily.
We will use Lefschetz' fixed point formula as in \eqref{eq:lef}.
Without further machinery it would require 
to count points as deep as down to the field $\F_{p^{10}}$ to find $\mu_p(T)$.
Presently we circumvent this by using the automorphism $\varphi$.
The crucial property is that $\varphi$ commutes with $\Fr_p$.
Hence the induced linear transformations on $\het{\bar X}$ can be diagonalised simultaneously,
and we can define the relative trace of Frobenius:
\[
a_i(q) = \tr {\Fr}^*_q(V_i) \;\;\; (i=0,\hdots,10).
\]
Note that $a_0(q)=2q$ always and $\tr {\Fr}^*_q(\het{\bar X}) = \sum_{i=0}^{10} a_i(q).$
Since $\varphi$ and $\Fr_p$ commute,
we have by definition
\[
\tr(\varphi^n\circ{\Fr}_q)^*(V_i) = \zeta^{ni}a_i(q).
\]
As a consequence, we find
\begin{eqnarray}
\label{eq:tr_n}
\;\;\; 
{\tr}_n(q) = \tr(\varphi^n\circ{\Fr}_q)^*(\het{\bar X}) = \sum_{i=0}^{10} \zeta^{ni}a_i(q) \;\;\; (n=0,\hdots,10).
\end{eqnarray}
Note that $\tr_n(q)\in\Z$ for any $n\in\{0,\hdots,10\}$,
and this integer can be computed by Lefschetz' fixed point formula again.
Conversely,
we can calculate each $a_i(q)$ from the collection of ${\tr}_n(q)$'s.

We continue by discussing how to efficiently compute ${\tr}_n(q)$.
Lefschetz' fixed point formula gives
\[
\# \mbox{Fix}(\varphi^n\circ{\Fr}_q) = 1 + {\tr}_n(q) + q^2.
\]
For $n=0$, the fixed locus is simply $X(\F_q)$.
We shall now analyse the fixed locus for ${n>0}$
and reduce the computations again to the field $\F_q$.
Clearly $\mbox{Fix}(\varphi^n\circ{\Fr}_q)$ contains
$q+1$ rational points on the zero section 
and at the fibre at $\infty$ exactly the $\F_q$-rational points;
these contribute $2q+1$ rational points in total.
It remains to study the fixed locus on the affine chart given in \eqref{eq:Xe} or \eqref{eq:gamma}.
Suppose that $(x,y,t)$ is a fixed point of $\varphi^n\circ {\Fr}_q$.
Keeping in mind $n>0$, the resulting conditions are first of all
\[
x^q=x,
y^q=y,
t^q+n=t
\;\;
{\Longrightarrow}
\;\;
x,y\in\F_q, t\not\in\F_q.
\]
But then we infer from the defining equation (over $\F_q$)
that $t^p-t=c\in\F_q$,
and any such $c$ occurs for exactly $p$ fibres.
Directly this implies
\begin{eqnarray}
\label{eq:c}
-n = t^q-t=c+c^p+\hdots+c^{p^{r-1}}={\tr}_{\F_q/\F_p} c.
\end{eqnarray}
We are now ready to compute all the quantities $\# \mbox{Fix}(\varphi^n\circ{\Fr}_q)$ at once.
Fix $\eps$ or $\gamma$ and pick any $(x,y)\in\F_q^2$.
Then this determines $c=t^p-t$ by \eqref{eq:Xe} or \eqref{eq:gamma}.
Compute $n$ by \eqref{eq:c}.
Then the affine point $(x,y)\in\F_q^2$ contributes exactly $p$-times to $\# \mbox{Fix}(\varphi^n\circ{\Fr}_q)$,
namely once for each fibre $F_t$ at a root of $t^p-t=c$.
Thus we can easily distribute all $(x,y)\in\F_q^2$ over all the fixed loci $ \mbox{Fix}(\varphi^n\circ{\Fr}_q)$.

\subsection{Execution}
\label{ss:exe}

For $\eps, \gamma\in\F_p$,
the above procedure enables us to calculate the characteristic polynomials of Frobenius
on $\het{\bar X}$
solely from computations over $\F_p$ and $\F_{p^2}$.
Below we tabulate the results for $V(1)$,
i.e.~we give the characteristic polynomial 
\[
\tilde\mu_p(T) = \mu_p(pT)/p^{20}
\]
on the 20-dimensional vector space $V$
with Galois representation shifted so that all eigenvalues have absolute value $1$.
Due to the shape of the equation, the characteristic polynomials only depend on the property
whether respectively $\eps$  or $\gamma$ is a square mod $p$ or not.

$$
\begin{array}{c|l}
\hline
\eps & \tilde\mu(T)\\
\hline
1,3,4,5,9 & T^{20}+T^{19}+2 T^{18}+T^{17}+T^{16}+2 T^{15}+3 T^{14}+5 T^{13}+4 T^{12}+3 T^{11}\\
& \; +\frac{23}{11} T^{10}+3 T^9+4 T^8+5 T^7+2 T^5+3 T^6+T^4+T^3+2 T^2+T+1\\
\\
2,6,7,8,10& T^{20}-T^{19}+2 T^{18}-T^{17}+T^{16}-2 T^{15}+3 T^{14}-5 T^{13}+4 T^{12}-3 T^{11}\\
& \; +\frac{23}{11} T^{10}-3 T^9+4 T^8-5 T^7-2 T^5+3 T^6+T^4-T^3+2 T^2-T+1\\
\\
\hline
\gamma & \tilde\mu(T)\\
\hline
1,3,4,5,9 & 1-4T^2+10 T^4-18 T^6+25 T^8-\frac{307}{11} T^{10}\\
& \;\;\;\; +25 T^{12}-18 T^{14}+10 T^{16}-4 T^{18}+T^{20}\\
\\
2,6,7,8,10 & 1-5T^2+12 T^4-18 T^6+20 T^8-\frac{219}{11} T^{10}\\
& \;\;\;\; +20 T^{12}-18 T^{14}+12 T^{16}-5 T^{18}+T^{20}
\end{array}
$$

All four characteristic polynomials are irreducible over $\Q$.
Since they are not integral, they are not cyclotomic.
Hence there are no roots of unity as zeroes,
and we deduce that $\rho(X)=2$ by \eqref{eq:Tate}.
This proves Theorem \ref{thm} for the family $\{X_\varepsilon\}$ 
(or equivalently $\{X_\gamma\}$),
i.e.~for those K3 surfaces
with a jacobian elliptic fibration compatible with the automorphism of order $11$.

\subsection{Easy consequences}

Before continuing with the proof of Theorem \ref{thm} for all K3 surfaces,
we note two corollaries. The first is straight-forward:

\begin{Corollary}
Any elliptic K3 surface over $\C$ reducing mod $p$ to some $X_\eps$ or $X_\gamma$ 
for $\eps,\gamma\in\F_p^\times$ has Picard number $2$.
\end{Corollary}

For the second corollary,
we recall the notion of height for a K3 surface $X$ in positive characteristic.
Generally this is defined in terms of the formal Brauer group $\hat{\Br}(X)$
and ranges through $\{1,2,\hdots,10\}$ and $\infty$
where the last case was coined {\it supersingular} by Artin \cite{A}.
Note that in case of finite height $h$,
there is an inequality 
\[
\rho\leq b_2-2h.
\]
If the K3 surface is defined over some finite field,
then the height translates into arithmetic 
through the Newton polygon of the characteristic polynomial of Frobenius.
This collects the $\pi$-adic valuations of all eigenvalues
at some prime $\pi$ above $p$ in the splitting field of the characteristic polynomial.
For instance, height $1$ corresponds to the so-called {\it ordinary} case where Hodge and Newton polygon coincide,
i.e.~some eigenvalue is a $\pi$-adic unit.
Equivalently, ${\tr}\, {\Fr}^*_q(\het{\bar X}) \not\equiv 0\mod p$.
Presently we find that all pairs of complex conjugate eigenvalues on the $20$-dimensional subspace $V\subset \het{\bar X}$ have 
evaluations $9/10, 11/10$. Hence we get the following corollary:

\begin{Corollary}
\label{Cor:h}
The K3 surfaces  $X_\eps$ and $X_\gamma$ 
for $\eps,\gamma\in\F_p^\times$
have height $10$.
\end{Corollary}

%
%
%
%
%
%
%
%
%
%
%

It is instructive to note that while our K3 surfaces are not supersingular
(in either sense, $h=\infty$ or $\rho=22$),
in terms of height they are as close to being supersingular as possible.
We should also point out that we are not aware of any other explicit examples
of K3 surfaces of height $10$.

\subsection{Proof of Theorem \ref{thm}}
\label{ss:full}

Recall that K3 surfaces with an automorphism $\varphi$ of order $11$
come in families consisting of torsors over the family $\{X_\eps\}$.
In fact,  any such K3 surface $X$ comes with an elliptic fibration 
whose jacobian $\mbox{Jac}(X)$ equals $X_\eps$ for some $\eps\in k$.
Generally one has by \cite[Cor.~5.3.5]{CD}
\[
\rho(X) = \rho(\mbox{Jac}(X)).
\]
Hence the result from \ref{ss:exe}
that a generic member of the family $\{X_\eps\}$ has Picard number $2$
carries over directly to all other families.
This proves Theorem \ref{thm}.

\subsection{Remark}

We should like to emphasise the conceptual difference to the standard case
of symplectic automorphisms over $\C$.
Given a finite group $G$ of automorphisms acting symplectically on a complex K3 surface $X$,
 Nikulin proved in \cite{N} that there is an abstract lattice $\Omega_G$ of fairly large rank such that
\[
(H^2(X,\Z)^G)^\bot = \Omega_G.
\]
Since the transcendental lattice $T(X)$ is contained in $H^2(X,\Z)^G$ by assumption,
we deduce that $X$ has fairly large Picard number.
In contrast, 
symplectic automorphisms of order $11$
(necessarily wild) imply generally $\rho=2$
as we have seen.

\section{Other characteristics}
\label{s:p}

In \cite[Rem.~6.6]{DK}, the authors give an elliptic K3 surface over $\Q$ with a
non-symplectic automorphism of order $11$
and good reduction outside $\{2,3,11\}$.
They ask whether there is a K3 surface with an automorphism of order 11 in characteristic $2$ and $3$ as well.
Moreover they pose the problem whether there is a supersingular example in any possible characteristic,
i.e.~by \cite{Ny} whenever there is some $\nu$ such that $p^\nu\equiv -1\mod 11$.
Equivalently $p$ is a non-square in $\F_{11}$.
Here we present a uniform answer in terms of the extended Weierstrass form
\begin{eqnarray}
\label{eq:uni}
y^2 + xy = x^3+t^{11}.
\end{eqnarray}

\begin{Proposition}
\label{prop}
Let $k$ be an algebraically closed field
and $X$  the elliptic K3 surface given by the Weierstrass form \eqref{eq:uni} over $k$.
Then $X$ admits an automorphism of order $11$.
Moreover $X$ is supersingular with $\rho(X)=22$ whenever char$(k)$ either  equals $11$
or is a non-square modulo $11$.
\end{Proposition}

\begin{proof}
It is easily computed that \eqref{eq:uni} gives indeed
an elliptic K3 surface $X$ of discriminant $t^{11}(432\,t^{11}+1)$.
There are singular fibres of Kodaira type $II$ at $\infty$, $I_{11}$ at $0$ and, depending on $p=$ char$(k)$,
\begin{itemize}
\item
11 fibres of type $I_1$ at the roots of $432\,t^{11}+1$ if $p\neq 2,3,11$;
\item
1 fibre of type $I_{11}$ at $t=7$ if $p=11$;
\item
no other singular fibres if $p=2,3$
as there is wild ramification of index $11$ at $\infty$.
\end{itemize}
On the one hand, if $p\neq 11$, then the automorphism of order $11$ thus is given by $t\mapsto \zeta t$
for a primitive 11th root of unity $\zeta\in k$.
On the other hand, if $p=11$,
then we read off $\NS(X)=U+A_{10}^2$.
In particular, $\rho(X)=22$
and $X$ is a supersingular K3 surface of Artin invariant $\sigma=1$.
As such, it is unique up to isomorphism by \cite{Ogus}.
As in Section \ref{s:background}, $X$ is therefore isomorphic to $X_0$  from \eqref{eq:Xe}.
We conclude that $X$ admits an automorphism of order $11$.

To prove Proposition \ref{prop}
it remains to justify the claim about supersingularity for  non-squares $p$ modulo $11$.
For this purpose, we note that $X$ is a Delsarte surface
covered by the Fermat surface $S$ of degree $11$.
Writing affinely 
\[
S=\{u^{11}+v^{11}+w^{11}+1=0\}\subset\PP^3,
\]
the dominant rational map
is given by
\begin{eqnarray*}
S\;\;\;\;\; & \dasharrow & \;\;\;\;\;X\\
(u,v,w) & \mapsto & (x,y,t) = (-u^{11}v^{11},-u^{22}v^{11}, -wu^3v^2)
\end{eqnarray*}
Since $S$ is unirational over $\bar\F_p$ for any non-square $p$ modulo $11$
by \cite{KS}, so is $X$.
Thus the supersingularity of $X$ follows from Shioda's theorem \cite{Sh-unirat}.
\end{proof}

\begin{Remark}
For char$(k)=0$ or a non-zero square modulo $11$,
the argument with the covering Fermat surface can be used to show that $\rho(X)=2$
by \cite{Sh}.
\end{Remark}

\subsection*{Acknowledgements}
I thank Igor Dolgachev and JongHae Keum  for enlightening discussions
and the referee for his/her helpful comments.


\begin{thebibliography}{99}

\bibitem{A} Artin, M.: \emph{Supersingular $K3$ surfaces}, Ann.~Scient.~\'Ec.~Norm.~Sup.~(4) {\bf 7} (1974), 543--568.


\bibitem{ASD} Artin, M., Swinnerton-Dyer, P.: \emph{The Shafarevich-Tate conjecture for pencils of elliptic curves on $K3$ surfaces}, Invent.~Math.~{\bf 20} (1973), 249--266.
%
%
%
%
%

\bibitem{CD}
Cossec, F.~R., Dolgachev, I.~V.:
\emph{Enriques surfaces}. I. 
Progress in Math.~{\bf 76}. Birkh\"auser (1989).

%

\bibitem{DK}
Dolgachev, I.,
Keum, J.H.:
\emph{K3 surfaces with a symplectic automorphism of order 11},
JEMS {\bf 11} (2009), 799--818.

\bibitem{KS} Katsura, T., Shioda, T.: \emph{On Fermat varieties}, Tohoku Math.~J.~(2) {\bf 31} (1979), no.~1, 97--115.


%
%
%
%
%
%
%
%

\bibitem{N}
Nikulin, V.V.:
\emph{Finite groups of automorphisms of K\"ahlerian K3 surfaces}, 
Trudy
Moskov. Mat. Obshch. {\bf 38} (1979), 75--137. English translation: Trans. Moscow Math.Soc. {\bf 38}
(1980), 71--135.

%

\bibitem{Ny}
Nygaard, N.:
\emph{Higher de Rham-Witt complexes of supersingular K3 surfaces}, 
Compositio
Math. {\bf 42} (1980/81), 245--271.

\bibitem{Ogus} Ogus, A.:
\emph{Supersingular K3 crystals},
Journ\'ees de G\'eom\'etrie Alg\'ebrique de Rennes (Rennes 1978),
Vol.~II,  3--86, Ast\'erisque {\bf 64}, Soc.~Math.~France, Paris, 1979.

%
%
%
%
%
%
%
%
%
%
%
 
 \bibitem{Shimada} Shimada, I.:  
\emph{Transcendental lattices and supersingular reduction lattices of a singular $K3$ surface}, 
Trans AMS {\bf 361} (2009), 909--949.

\bibitem{Sh-unirat}
Shioda, T.:
\emph{An example of unirational surfaces in characteristic p},
Math. Ann. {\bf 211} (1974), 233--236. 

\bibitem{Sh}
Shioda, T.:
\emph{An explicit algorithm for computing the Picard number of certain algebraic surfaces},
Amer.~J.~Math.~{\bf 108}, No. 2 (1986), 415--432.
 
 \bibitem{shioda}
Shioda, T.:
{\it The Mordell-Weil lattice of $y^2=x^3 + t^5 - 1/t^5 -11$}, 
Comment. Math. Univ. St. Pauli {\bf 56} (2007), 45--70.

\bibitem{Tate} Tate, J.:
 {\it Algebraic cycles and poles of zeta functions},
in: {\it Arithmetical Algebraic Geometry}
(Proc. Conf. Purdue Univ., 1963), 93--110, Harper \& Row (1965).

%
%
%

\end{thebibliography}
\end{document}